\documentclass[12pt]{amsart}
\usepackage{amsmath,amssymb,amsthm}
\usepackage{mathrsfs}
\usepackage{color}
\usepackage{pgfplots}
\pgfplotsset{compat=1.15}
\usepackage{mathrsfs}
\usetikzlibrary{arrows}

\newtheorem{theorem}{Theorem}[section]

\newtheorem{proposition}[theorem]{Proposition}
\newtheorem{corollary}[theorem]{Corollary}
\newtheorem{lemma}[theorem]{Lemma}
\newtheorem{remark}[theorem]{Remark}

\newcommand{\bigo}[1]{\mathcal{O}\left(#1\right)}

\newcommand{\class}{\mathscr{C}}

\newcommand{\im}{\text{Im}\,}

\newcommand{\NN}{\mathbb{N}}
\newcommand{\ZZ}{\mathbb{Z}}
\newcommand{\QQ}{\mathbb{Q}}
\newcommand{\RR}{\mathbb{R}}

\newcommand{\black}[1]{\color{black}}

\renewcommand{\d}{\partial}
\newcommand{\E}{\mathscr{E}}

\begin{document}

\title[Rigidity of symplectic billiards]{Deformational spectral rigidity of axially symmetric symplectic billiards}

\author{Corentin Fierobe}
\address{Department of Mathematics, University of Rome Tor Vergata, Via della ricerca scientifica 1, 00133 Rome, Italy}
\email{cpef@gmx.de}

\author{Alfonso Sorrentino}
\address{Department of Mathematics, University of Rome Tor Vergata, Via della ricerca scientifica 1, 00133 Rome, Italy}
\email{sorrentino@mat.uniroma2.it}

\author{Amir Vig}
\address{Department of Mathematics, University of Michigan, Ann Arbor, MI 48109}
\email{vig@umich.edu}

\maketitle


\begin{abstract}
    Symplectic billiards are discrete dynamical systems which were introduced by Albers and Tabachnikov and take place in a strictly convex bounded planar domain with smooth boundary. They are described by the \textit{symplectic law of reflection}, in constrast to the elastic reflection law of Birkhoff billiards. In this paper, we prove a version of dynamical spectral rigidity for symplectic billiards which is a counterpart to previous results on classical billiards by De Simoi, Kaloshin and Wei. Namely, we show that close to an ellipse, any sufficiently smooth one-parameter family of axially symmetric domains either contains domains with different area spectra or is trivial, in the sense that the domains differ by area-preserving affine transformations of the plane. We also prove that in general -- that is, even if the domains are not close to an ellipse -- any sufficiently smooth one-parameter family of axially symmetric domains which preserves the area-spectrum is tangent to a finite dimensional space.
\end{abstract}

\section{Introduction}

%

In this paper we consider \textit{symplectic billiards}, which were introduced by Albers and Tabachnikov \cite{AlbersTabachnikov}. These systems involve bounded, strictly convex, planar domains with smooth boundaries. The symplectic billiard map concatenates line segments according to the \textit{symplectic law of reflection}. This law can be described as follows (see \cite{AlbersTabachnikov} for more details): if a particle is emitted from a point $q_1$ on the boundary of the domain and strikes the boundary again at $q_2$, then following the line segment $q_1q_2$, it will bounce at $q_2$ and continue along a new trajectory until it hits the boundary at point $q_3$. The point $q_3$ is uniquely determined by the condition that the line $q_1q_3$ is parallel to the tangent line to the boundary at $q_2$ (see Figure \ref{figure:symplectic_bounce}).
\\
\\
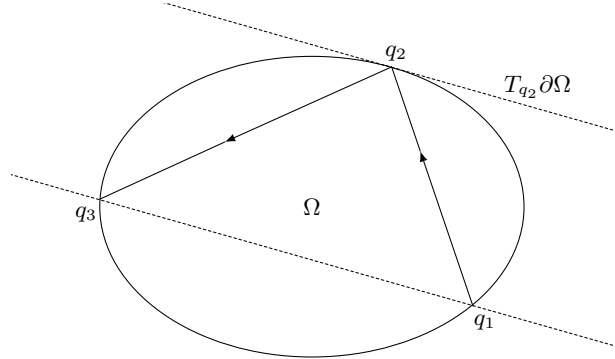
\begin{figure}[!ht]
\begin{tikzpicture}[line cap=round,line join=round,>=triangle 45,x=2.0cm,y=2.0cm]
\clip(-2,-1.2) rectangle (2,1.35);
\draw [rotate around={0:(0,0)}] (0,0) ellipse (2.82cm and 2cm);
\draw [dash pattern=on 1pt off 1pt,domain=-3.19:3.33] plot(\x,{(--32-8.43*\x)/29.69});
\draw [dash pattern=on 1pt off 1pt,domain=-3.19:3.33] plot(\x,{(-10.52-8.43*\x)/29.69});
\draw [-latex] (1.07,-0.66) -- (0.72,0.37);
\draw [-latex] (0.53,0.93) -- (-0.58,0.43);
\draw (-0.58,0.43)-- (-1.41,0.05);
\draw (0.72,0.37)-- (0.53,0.93);
\begin{scriptsize}
\draw[color=black] (1.15,-0.77) node {$q_1$};
\draw[color=black] (0.56,1.02) node {$q_2$};
\draw[color=black] (-1.5,-0.05) node {$q_3$};
\draw[color=black] (1.5,0.8) node {$T_{q_2}\partial\Omega$};
\draw[color=black] (0,0) node {$\Omega$};
\end{scriptsize}
\end{tikzpicture}
\label{figure:symplectic_bounce}
\caption{A symplectic bounce at $q_2$ in a domain $\Omega$. The line $q_1q_3$ and the tangent line to $\partial\Omega$ at $q_2$ are parallel.}
\end{figure}

The term \textit{symplectic billiard} is derived from the fact that the dynamics of such systems can be described by a variational principle related to standard symplectic (area) form. Let $\omega$ denote the standard symplectic form on the plane, and let $q_1, q_2, q_3$ be distinct points on the domain's boundary. A symplectic billiard bounce from $q_1$ to $q_2$ to $q_3$, as previously described, occurs if and only if $q_2$ is a critical point of the quantity
$$\omega(q_1,q_2)+\omega(q_2,q_3).$$
More precisely, the billiard map associated to symplectic billiards is an exact-symplectic twist-map of an annulus, whose generating function is given by $\omega$ \cite{AlbersTabachnikov}.\\

\subsection{Main results}
In this paper, we establish a rigidity result for symplectic billiards, described as follows. Corresponding to symplectic billiards in a domain $\Omega$, we define the area spectrum $\mathcal{A}(\Omega)$ to be the closure of the set of actions of its periodic trajectories. This area spectrum  is related to the set of areas enclosed by periodic trajectories. Due to the symplectic structure, the area spectrum remains invariant if we apply to the domain an affine transformations preserving the area \cite{AlbersTabachnikov}. The inverse problem asks whether or not it is possible to uniquely determine a domain $\Omega$, up to affine transformations, from its area spectrum. Specifically, we provide a partial answer to the following question:

\vspace{0.2cm}
\textbf{Question.} {\it Let $I$ be an open interval containing $0$ and $(\Omega_{\tau})_{\tau\in I}$ be a one-parameter family of domains such that  $\mathcal{A}(\Omega_{\tau})=\mathcal{A}(\Omega_{0})$ for any $\tau\in I$. \textit{Is it true that $\Omega_{\tau}=f_{\tau}(\Omega_0)$ where $f_{\tau}$ is an affine transformation of the plane?}}\\

\vspace{0.2cm}
A deformation $(\Omega_{\tau})_{\tau\in I}$ of $\Omega_0$ is said to be \textit{affine} if for any $\tau\in I$, there is an affine transformation $f_{\tau}:\RR^2\to\RR^2$ preserving the area such that $\Omega_{\tau}=f_{\tau}(\Omega_0)$. If $\mathcal D$ is a subset of all possible domains  and $\Omega\in \mathcal D$, we say that $\Omega$ is \textit{area spectrally rigid in $\mathcal{D}$} if the answer to previous question is positive for any sufficiently smooth family $(\Omega_{\tau})_{\tau\in I}$ with $\Omega_0=\Omega$ and $\Omega_{\tau}\in\mathcal D$ for $\tau\in I$. More precisely, we show the following result:

\begin{theorem}[Perturbative result]
\label{theorem:main1_simple}
Let $\mathcal S$ be the set of bounded axially symmetric strictly convex domain with $\mathscr C^7$-smooth boundary. Then any domain $\Omega\in\mathcal S$ which is sufficiently $\mathscr C^7$-close to an ellipse is area spectrally rigid in $\mathcal{S}$.
\end{theorem}

We refer to Theorem \ref{theorem:main1_simple} as \textit{perturbative} because it applies to domains which are \textit{close} to ellipses in a suitable topology. For more general domains, we can still provide a condition on the one-parameter families $(\Omega_{\tau}){\tau\in I}$ in $\mathcal{S}$ for which the area spectrum $\mathcal A(\Omega{\tau})$ remains constant.

\begin{remark}
    The proof relies on the existence of symmetric periodic orbits of every rotation number of the type $1/q$ with $q\geq 3$. We believe that a more general statement can be proven by considering other rotation numbers.
\end{remark}

Assume that the deformation $\Omega_\tau$ is $\mathscr{C}^7$-smooth in $\tau$. Such a family can be described using a collection of $\mathscr{C}^6$-smooth maps, known as \textit{deformation maps}, denoted by $(n_{\tau})_{\tau}$. For each $\tau$, the map $n_{\tau}: \partial\Omega_{\tau} \to \RR$ assigns to each point $q_{\tau} \in \partial\Omega_{\tau}$ a value that approximately represents the area of the parallelogram formed by the unit tangent vector to $\partial\Omega_{\tau}$ at $q_{\tau}$ and the vector $\partial_{\tau}q_{\tau}$.

\begin{theorem}[Non perturbative result]
\label{theorem:main2_simple}
Let $(\Omega_{\tau})_{\tau\in I}$ be a one-parameter family of domains in $\mathcal{S}$ such that $\mathcal A(\Omega_{\tau})=\mathcal A(\Omega_{0})$ for all $\tau$. Then, there is a continuous family $(V_{\tau})_{\tau}$ of finite dimensional vector spaces $V_{\tau}$ of the space $\mathscr C^6(\partial\Omega_{\tau},\RR)$ such that $n_{\tau}\in V_{\tau}$. Moreover, $V_0$ is uniquely determined by $\Omega_0$.
\end{theorem}

\smallskip

\subsection{Background}
This work was inspired by a breakthrough by De Simoi, Kaloshin, and Wei \cite{DKW} in the context of classical Birkhoff billiards, namely, billiard models in which the reflection law at impact points is given by the optical law: the angle of incidence equals the angle of reflection. In that case, isospectrality is with respect to the \textit{lengths} or periodic orbit, as opposed to areas. One also studies the spectrum of the Laplacian with suitable (e.g., Dirichlet) boundary conditions on a billiard table. In 1967, Mark Kac asked the famous question \textit{``Can one hear the shape of a drum?''}, i.e., are Laplace isospectral sets (modulo Euclidean isometries) multiplicity free? While it has been shown that the general answer to this question is negative \cite{GordonWebbWolpert}, it remains unresolved for the class of strictly convex domains with smooth boundaries. Osgood, Phillips, and Sarnak \cite{OsgoodPhillipsSarnak1,OsgoodPhillipsSarnak2,OsgoodPhillipsSarnak3} demonstrated that Laplace-isospectral sets of planar domains are compact in the $\mathscr{C}^{\infty}$ topology. The third author proved an analogous result for the marked length spectrum \cite{Vig}. In the setting of one-parameter families of domains, Hezari and Zelditch \cite{HezariZelditch} provided a positive answer for Laplace isospectral analytic deformations of ellipses that preserve two axes of symmetry, as well as a strong condition for such deformations in the $\mathscr{C}^{\infty}$ setting. These results were further extended by Popov and Topalov \cite{PopovTopalov1}.

%

\subsection{Outline of the paper}
The paper is organized as follows. In Section \ref{section:def_statements} we introduce more precisely the symplectic billiard map as well as its area spectrum, and we state our main theorems in more accurate terms -- see Theorems \ref{theorem:main1_simple} and \ref{theorem:main2_simple}. Section \ref{section:construct_isospectral} is devoted to the construction of an operator which we  call \textit{linear isospectral operator} -- following the terminology in \cite{DKW}: the latter is associated to a specific domain and its injectivity is related to the rigidity of the corresponding domain. We prove our main Theorems \ref{theorem:main1_simple} and \ref{theorem:main2_simple} using the key estimates given by Proposition \ref{proposition:perturbative_regime}. Appendix \ref{section:affine_parametrization} recalls notions of affine parametrizations of planar curves. Appendix \ref{section:billiard_ellipse} presents explicit formulae for periodic orbits of the symplectic billiard in an ellipse. Appendix \ref{section:proof_estimates} gives asymptotic estimates for periodic orbits of the symplectic billiard in a domain sufficiently $\mathscr C^7$-close to an ellipse; these estimates generalize estimates which can be found in \cite{BaraccoBernardiNardi}. Appendix \ref{section:operators_appendix} states invertibility properties of operators acting on Sobolev spaces.

\subsection{Acknowledgements}
AS acknowledges the support of the Italian Ministry of University and Research’s PRIN 2022 grant “Stability in Hamiltonian dynamics and beyond”, as well as the Department of Excellence grant MatMod@TOV (2023-27) awarded to the Department of Mathematics of University of Rome Tor Vergata. AS is a member of the INdAM research group GNAMPA and the UMI group DinAmicI. Concurrently with this work, the authors of \cite{BaraccoBernardiNardi} arrived at a similar result using different techniques. We shared our ideas and mutually acknowledged that these results were obtained independently from one another.

\section{Definitions and statements}
\label{section:def_statements}

\subsection{Symplectic billiard map}
Let $\Omega$ be a bounded strictly convex planar domain and $\gamma: \RR/L \ZZ \longrightarrow \RR^2$ be a parametrization of  $\partial \Omega$, where $L>0$.
Given a  point $\gamma(t)$, denote by $\gamma(t^*)$  the other point on $\partial \Omega$ where the tangent line is parallel to that of $\gamma(t)$. 
According to \cite{AlbersTabachnikov}, the phase space of the symplectic billiard map in $\Omega$ is then the set of the oriented chords $\gamma(t_0) \gamma(t_1)$ where $t_0 < t_1 <t_0^*$, with respect to the orientation of $\gamma$; alternatively, the phase space can be described as the set 
$$ \mathcal X_\Omega:= \{(t_0, t_1): \; \omega(\dot \gamma(t_0), \dot \gamma (t_1)) >0\}$$ 
where $\omega$ denotes the standard area form in the plane ({\it i.e.}, determinant of the matrix made by the two vectors). The vertical foliation consists of all chords with a fixed initial point. The symplectic billiard map is given by:
\begin{eqnarray}\label{Bsymmap}
B_\Omega: \mathcal X_\Omega &\longrightarrow &  \mathcal X_\Omega \nonumber,\\
(t_0,t_1) &\longmapsto& (t_1,t_2),
\end{eqnarray}
where $(t_1,t_2)$ is uniquely determined by the condition that the tangent line to $\partial \Omega$ at $\gamma(t_1)$ is parallel to the line $\gamma(t_0)\gamma(t_2)$ (see Figure \ref{figure:symplectic_bounce}). The map $B_\Omega$ is an exact symplectic twist map and can be extended to the boundary of its phase space by continuity: $B_\Omega (t,t) := (t,t)$ and 
$B_\Omega (t,t^*) := (t^*,t)$.

\begin{remark}
\label{remark:affine_transformation_commute}
The symplectic billiard map $B_\Omega$ can also be defined on the set of points $(q_0,q_1)\in\partial\Omega\times\partial\Omega$. In this setting, it is known that it commutes with affine transformations: if $F$ is such a transformation $B_{F(\Omega)}(F(q_0,q_1))=F(B_{\Omega}(q_0,q_1))$ where $F(q_0,q_1)$ stands for $(F(q_0),F(q_1))$.
\end{remark}

\subsection{Generating function and area spectrum}
The symplectic billiard map is an exact-symplectic twist map and hence, is associated to a generating function $S_\Omega:\{(t_0, t_1): \; t_0<t_1<t_0^*\}\to\RR$ defined for all $(t_0, t_1)$ by
$$
S_\Omega(t_0, t_1) = \omega(\gamma(t_0),\gamma(t_1)).
$$
Here, $\omega = dx\wedge dy$ is the standard area form of $\RR^2(x,y)$. It has the property that given three distinct points $t_0<t_1<t_2$, $B_\Omega (t_0,t_1) = (t_1,t_2)$ if and only if 
\begin{equation}
    \label{equation:variational_property}
\frac{d}{dt_1}\left(S_{\Omega}(t_0,t_1)+S_{\Omega} (t_1,t_2)\right)=0.
\end{equation}
A sequence $\underline t = \{(t_k,t_{k+1})\}_k$ such that for any integer $k$ we have $B_\Omega (t_k,t_{k+1})=(t_{k+1},t_{k+2})$ and $(t_q,t_{q+1})=(t_0,t_1)$ is called a \textit{periodic orbit} of $B_\Omega$ of period $q>0$. We define its action by
$$A(\underline t) = \sum_{k=0}^{q-1}S_\Omega(t_k,t_{k+1}).$$
When the points $\gamma(t_0),\ldots,\gamma(t_{q-1})$ form a polygon (\textit{i.e.} without intersection of edges), then $A(\underline t)$ corresponds to twice its area. We consider the \textit{area spectrum} of $\Omega$, denoted by $\mathcal A(\Omega)$, which is defined to be the \textit{closure of the set}
$$\{A(\underline t)\;|\; \underline t\text{ is a periodic orbit of }\Omega\}.$$

\begin{remark}
Following Remark \ref{remark:affine_transformation_commute}, if $\Omega'$ is a domain obtained from $\Omega$ by applying an area preserving affine transformation, then each periodic orbit of $\Omega$ is transformed into a periodic orbit of $\Omega'$ with the same action, and necessarily the area spectrum of both domains coincide, namely $\mathcal A(\Omega')=\mathcal A(\Omega).$
\end{remark}

We say that a periodic orbit $\underline t$ has \textit{rotation number} $p/q\in\QQ$ if for any integer $k$, it satisfies $t_{k+q} = t_k+p$ when lifted to the universal cover (a strip). Given an integer $q\geq 2$, let $A_q$ be the maximal area of a $q$-periodic orbit having rotation number $1/q$ in $\Omega$. It was proven \cite{AlbersTabachnikov} that there exists a sequence $(a_n)_{n\geq 0}$ of real numbers such that for any $q\geq 2$
$$A_q = a_0+\frac{a_1}{q^2}+\frac{a_2}{q^4}+\ldots$$
where in particular $a_0$ is the area of $\Omega$,
$$a_1 = \frac{L^3}{12}\qquad\text{and}\qquad a_2 = -\frac{L^4}{240}\int_0^L k(t)dt,$$
where $k$ is the affine curvature of $\partial\Omega$ and $L$ its affine-perimeter. The latter are defined more precisely in Section \ref{section:affine_parametrization}.

\usetikzlibrary{decorations.markings}
\usetikzlibrary{decorations.pathreplacing}
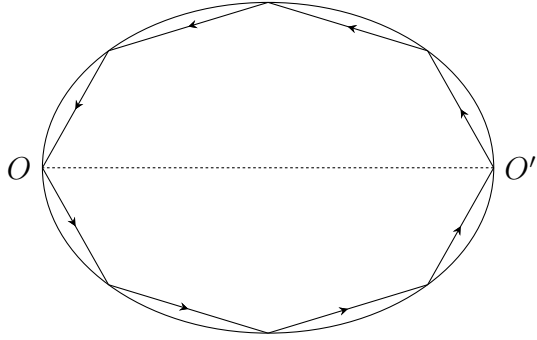
\begin{figure}
	\begin{tikzpicture}[xscale=0.75, yscale =1.1, decoration={
			markings,
			mark=at position 0.5 with {\arrow{>}}}
		] 
		\def\axisX{4} 
		\def\axisY{2} 
		
		\draw[-latex] (0,0) ellipse (\axisX cm and \axisY cm);
		
		\foreach \i in {1,...,8} {
			\coordinate (P\i) at ({\axisX * cos(45 * (\i - 1))}, {\axisY * sin(45 * (\i - 1))});
		}
		
		\draw[-, >=stealth, decoration={markings, mark=at position 0.5 with {\arrow{>}}}, postaction={decorate}] (P1) -- (P2);
		
		\draw[-, >=stealth, decoration={markings, mark=at position 0.5 with {\arrow{>}}}, postaction={decorate}] (P2) -- (P3);
		
		\draw[-, >=stealth, decoration={markings, mark=at position 0.5 with {\arrow{>}}}, postaction={decorate}] (P3) -- (P4);
		
		\draw[-, >=stealth, decoration={markings, mark=at position 0.5 with {\arrow{>}}}, postaction={decorate}] (P4) -- (P5);
		
		\draw[-, >=stealth, decoration={markings, mark=at position 0.5 with {\arrow{>}}}, postaction={decorate}] (P5) -- (P6);
		
		\draw[-, >=stealth, decoration={markings, mark=at position 0.5 with {\arrow{>}}}, postaction={decorate}] (P6) -- (P7);
		
		\draw[-, >=stealth, decoration={markings, mark=at position 0.5 with {\arrow{>}}}, postaction={decorate}] (P7) -- (P8);
		
		\draw[-, >=stealth, decoration={markings, mark=at position 0.5 with {\arrow{>}}}, postaction={decorate}] (P8) -- (P1);
		
		\draw[dash pattern=on 1pt off 1pt] (P1) -- (P5);

        \node[right] at (P1) {$O'$};  
        
        \node[left] at (P5) {$O$}; 
        
	\end{tikzpicture}
	\label{AxiallySymmetricOctagon}
	\caption{An axially symmetric symplectic billiard orbit of rotation number $\omega = 1/8$.}
\end{figure}

\section{Construction of an isospectral operator}
\label{section:construct_isospectral}

Let $(\Omega_{\tau})_{\tau\in I}$ be an axially symmetric $\class^7$-smooth deformation of $\Omega=\Omega_0$. By applying affine transformations to each domain in this family, we may assume that there are two distinct points $O$ and $O'$ on the plane such that for each $\tau$, the boundary $\partial\Omega_{\tau}$ contains both $O$ and $O'$ and the line $OO'$ is an axis of symmetry of $\Omega_{\tau}$ (Figure \ref{AxiallySymmetricOctagon}). Hence, we consider a $\mathscr C^7$-smooth map $\gamma:I\times\RR\to\RR^2$ such that for any $\tau\in I$, $\gamma(\tau,\cdot)$ parametrizes $\partial{\Omega_{\tau}}$ with $\gamma(\tau,0)=O$ and $\gamma(\tau,-t)=\gamma(\tau,t)$. Throughout the paper, we will assume that for each $\tau\in I$, the map $\gamma(\tau,\cdot)$ is an affine parametrization of $\partial\Omega_{\tau}$; i.e., $\det(\partial_t\gamma(\tau,t),\partial_t^2\gamma(\tau,t))=1$ for any $\tau\in I$ and $t\in \RR$  -- see Appendix \ref{section:affine_parametrization} for a discussion of affine parametrizations. Moreover if $L_{\tau}$ is the affine perimeter of $\partial\Omega_{\tau}$, we demand that $\gamma(\tau,0)=O$ and $\gamma(\tau,L_{\tau}/2)=O'$. 

\subsection{Deformation map} 
\label{subsection:deformation_area}

Given a family of parameters $\underline t = (t_0,\ldots,t_{q-1})$ and $\tau\in I$, we define the action
$$A_q(\tau,\underline t) = \sum_{k=0}^{q-1}\omega(\gamma(\tau,t_k),\gamma(\tau,t_{k+1})).$$
We also introduce the \textit{deformation map} of the family $(\Omega_{\tau})_{\tau\in I}$ as the quantity $n_{\tau}(s)$ defined for any $\tau\in I$ and $s\in\RR$ by
$$
n_{\tau}(s) = \omega(\partial_{\tau}\gamma(\tau,s),\partial_{s}\gamma(\tau,s)).
$$
It satisfies the following property:

\begin{lemma}
\label{lemma:differential_of_area}
Let $\tau_0\in I$ and assume that $(\gamma(\tau_0,t_k))_k$ corresponds to a $q$-periodic  symplectic billiard orbit in $\Omega_{\tau_0}$. Then
$$
\partial_{\tau} A_q(\tau_0,\underline t) = \sum_{k=0}^{q-1}n_{\tau_0}(t_k)\varrho_{\tau_0}^{-1/3}(t_k)L_k,
$$
where $\varrho_{\tau}(t)$ is the radius of curvature of $\partial\Omega_{\tau}$ at $\gamma(\tau,t)$, $L_k$ is the length of the segment between the two points $\gamma(\tau_0,t_{k-1})$ and $\gamma(\tau_0,t_{k+1})$, and $n_{\tau}(s)$ is the \textit{deformation map} of the family $(\Omega_{\tau})_{\tau\in I}$.
\end{lemma}

\begin{proof}
Using bilinearity of the action, we see that
$$\partial_{\tau} A_q(\tau_0,\underline t) = \sum_k \omega(\partial_{\tau}\gamma(\tau_0,t_k),\gamma(\tau_0,t_{k+1}))+\omega(\gamma(\tau_0,t_k),\partial_{\tau}\gamma(\tau_0,t_{k+1})).
$$
Changing index summation in the second term, we obtain
$$\partial_{\tau} A_q(\tau_0,\underline t) = 
\sum_k \omega(\partial_{\tau}\gamma(\tau_0,t_k),\gamma(\tau_0,t_{k+1})-\gamma(\tau_0,t_{k-1})).$$
Now, by definition of the reflection law,
$$
\gamma(\tau_0,t_{k+1})-\gamma(\tau_0,t_{k-1}) = \alpha_k \partial_s\gamma(\tau_0,t_{k}),
$$
where $\alpha_k>0$. Thus, we conclude that 
$$\alpha_k = \|\gamma(\tau_0,t_{k+1})-\gamma(\tau_0,t_{k-1})\|\|\partial_s\gamma(\tau_0,t_{k})\|^{-1} = \varrho_{\tau_0}^{-1/3}(t_k)L_k$$
since $\|\partial_s\gamma(\tau,t_{k})\| = \varrho_{\tau}^{1/3}(t_k)$.
\end{proof}

\begin{proposition}
\label{proposition:zero_deformation_map}
Assume that the deformation map $n_{\tau}(s)$ vanishes for any $(\tau,s)$. Then $\Omega_{\tau}=\Omega_0$ for any $\tau$.
\end{proposition}

\begin{proof}
This is classical and follows immediately from the paper \cite{DKW}. The fact that $n\equiv0$ implies that for any $(s,\tau)$ the vectors $\partial_{\tau}\gamma(\tau,s)$ and $\partial_{s}\gamma(\tau,s)$ are collinear. In other words, $D\gamma(s,\tau)$ has rank one. By the constant rank theorem, $\im\gamma$ is contained in $\partial\Omega_0$, whence $\partial \Omega_0=\partial\Omega_{\tau}$ for any $\tau$.
\end{proof}

\subsection{A family of nearly glancing orbits} Let $\Omega$ be a strictly convex axially symmetric domain with $\class^1$-smooth boundary and $q\geq 2$ an integer. Assume that one of the two points of $\partial\Omega$ on the symmetry axis denoted by $O$ is marked and call it the origin. Consider the set $\mathcal P^{(q)}$ of all convex $q$-gons
$$
\underline p^{(q)}=(p_0^{(q)},\ldots,p_{q-1}^{(q)})
$$
inscribed in $\Omega$ which are symmetric with respect to the axis of symmetry, with $p_0^{(q)}$ fixed at the origin.

\begin{proposition}
\label{proposition:symmetric_maximal_orbit}
Let $\underline p^{(\Omega,q)}$ be the polygon maximizing the area of all polygons in $\mathcal P^{(q)}$. Then $\underline p^{(\Omega,q)}$ is a $q$-periodic orbit of the symplectic billiard in $\Omega$.
\end{proposition}

\begin{proof}
    Let $\underline{t} = (t_1\ldots,t_{q-1})$ be such that $p_j^{(q)}=\gamma(t_j)$ is the $j$-th point of the polygon $\underline p^{(\Omega,q)}$ for any $j\in\{0,\ldots,q-1\}$. Because $\underline p^{(\Omega,q)}$ is maximizing, we obtain that $\underline{t}$ is a critical point of the functional
    \begin{equation}
        \label{equation:functional_max}
    \underline{t}\mapsto \sum_{j=0}^{q-1}S_{\Omega}(t_j,t_{j+1})
    \end{equation}
    where $t_0=t_q=0$. For $j\in\{1,\ldots,q-1\}$, differentiating \eqref{equation:functional_max} with respect to $t_j$ induces that $(p_{j-1},p_{j})$ is sent to $(p_{j},p_{j+1})$ by the symplectic billiard map. It is also true for $j=0$ as a consequence of the symmetry assumption.
\end{proof}

Given a map $n:\RR\to\RR$ and an integer $q\geq2$, we define its \textit{discrete $X$-ray transform} along the orbit $\underline t^{(q)}$ to be the quantity $a_{\Omega,q}(n)\in\RR$, where
\begin{equation}
\label{equation:a_q}
a_{\Omega,q}(n) = \sum_{k=0}^{q-1}n\left(t_k^{(q)}\right)\varrho^{-1/3}(t_k^{(q)})L_{k}^{(q)}
\end{equation}
where $t_k^{(q)}$ is the affine arclength coordinate of $p_k^{(q)}$ on $\partial\Omega$, $\varrho(t)$ is the radius of curvature of $\partial\Omega$ at a point of affine arclength $t$ and $L_{k}^{(q)}$ is the Euclidean distance between $p_{k-1}^{(q)}$ and $p_{k+1}^{(q)}$.

\begin{remark}
Note that in the case when $q=2$, we have the trivial identity $a_{\Omega,2}(n)=0$ for any map $n$, as $L_0^{(2)}=L_1^{(2)}=0$. This follows from the fact that $2$-periodic symplectic orbits have zero action.
\end{remark}

\begin{proposition}
\label{proposition:vanish_on_n}
Assume that $\mathcal A(\Omega_{\tau})=\mathcal A(\Omega_{0})$ for any $\tau\in I$. Then for any $\tau \in I$ and any $q\geq 2$,
$$
a_{\Omega_{\tau},q}\left(n_{\tau}\right) = 0,
$$
where $n_{\tau}$ is the deformation area. Moreover, the affine perimeter $L_{\tau}$ of $\partial\Omega_{\tau}$ is independent of $\tau$. I.e., for any $\tau \in I$,
$$
L_{\tau} = L_{0}.
$$
\end{proposition}

\begin{proof}
We observe as in \cite{DKW} that the set $\mathcal A(\Omega_{0})$ has measure zero. We describe here an argument found in non-published lecture notes of De Simoi. Let $\Omega$ be a symplectic billiard with a $\mathscr C^r$-smooth boundary, $r\geq 2$. Define the maps $\Theta_q:\RR^2\to\RR^{q}$  and $\mathscr A:\RR^q\to\RR$ as follows: 1) $\Theta_q$ is the $\mathscr C^{r-1}$-smooth map which associates to a pair $(t_0,t_1)\in\RR^2$ the $q$-uple $(t_0,t_1,\ldots,t_{q-1})\in\RR^q$ such that the symplectic billiard map sends $(t_k,t_{k+1})$ to $(t_{k+1},t_{k+2})$ for any $k\in\{0,\ldots,q-2\}$; 2) $\mathscr A$ is the $\mathscr C^{r}$-smooth map defined by $\mathscr A(t_0,\ldots,t_{q-1})= \sum_{j=0}^{q-1} S_{\Omega}(t_j,t_{j+1})$ for any $q$-uple $(t_0,t_1,\ldots,t_{q-1})\in\RR^q$. It follows from the variational property \eqref{equation:variational_property} that the action of a periodic orbit is a critical value of the $\mathscr C^{r-1}$-smooth map of two variables $\mathscr A\circ\Theta:\RR^2\to\RR$. Hence if $r-1>1$, \textit{i.e.} if $r>2$, by Sard's theorem the critical values of $\mathscr A\circ\Theta$ has measure zero, whence the same holds for the area spectrum $\mathcal A(\Omega_{0})$.
\\
\\
For $\tau\in I$ and $q\geq 2$, denote by
$$
\underline p^{(\Omega_{\tau},q)}=(p_0^{(q,\tau)},\ldots,p_{q-1}^{(q,\tau)})
$$
the maximal symplectic billiard orbit in $\Omega_{\tau}$ which was introduced in Proposition \ref{proposition:symmetric_maximal_orbit}. The map $\tau\mapsto A_q(\tau,\underline t^{(q,\tau)})\in\mathcal A(\Omega_{\tau})=\mathcal A(\Omega_{0})$ is continuous and hence, constant. We denote this constant by $\alpha_q$. Let us fix a $\tau_0\in I$. Since each polygon $\underline p^{(q,\tau)}$ maximizes the area, for any $\tau\in I$, we have 
$$
A_q(\tau,\underline t^{(q,\tau_0)})\leq \alpha_q = A_q(\tau,\underline t^{(q,\tau)}),
$$
with equality at $\tau=\tau_0$. The left-hand side is differentiable in $\tau$ and admits a maximal value at $\tau=\tau_0$. 
Hence $\partial_{\tau}A_q(\tau_0,\underline t^{(q,\tau_0)})=0$. By Lemma \ref{lemma:differential_of_area}, the latter quantity is $a_{\Omega_{\tau_0},q}\left(n_{\tau_0}\right)$.
Moreover, for any $\tau\in I$, it was shown in \cite{AlbersTabachnikov} that
$$
A_q(\tau,\underline t^{(q,\tau)}) \sim \frac{L_{\tau}^3}{12}
$$
as $q \to \infty$. Since $A_q(0,\underline t^{(q,0)}) = A_q(\tau,\underline t^{(q,\tau)})$, we conclude that $L_0=L_{\tau}$.
\end{proof}

\subsection{Domains close to ellipses} The quantities $a_{\Omega,q}(n)$ can be computed explicitly when $\Omega$ is an ellipse $\mathscr E$. Moreover when $\Omega$ is close to an ellipse $\mathscr E$ in the $\mathscr C^7$ topology, $a_{\Omega,q}(n)$ is well-approximated by $a_{\mathscr E,q}(n)$. These estimates are presented in this section, namely in Propositions \ref{proposition:elliptical_variation} and \ref{proposition:perturbative_regime}.

Let $\Omega$ be a strictly convex axially symmetric domain with $\class^7$-smooth boundary which is $\delta$-close to an ellipse $\mathscr E$ in the $\class^7$-topology for some $\delta>0$.
\\
\\
\begin{proposition}[Elliptical symplectic billiards]
\label{proposition:elliptical_variation}
If the ellipse $\mathscr E$ has affine curvature $k_{\mathscr E}$, then $a_{\mathscr E,q}(n)$ can be expressed for any $q\geq 2$ and any map $n:\RR\to\RR$ as
$$a_{\mathscr E,q}(n)=\mu_q[n]_q,$$
where $\mu_q = 2k_{\mathscr E}^{-1/2}q\sin(2\pi/q)$ and $[n]_q$ denotes the cyclic sum
$$[n]_q = \frac{1}{q}\sum_{j=0}^{q-1}n(j/q).$$
\end{proposition}
Any smooth, $1$-periodic, even map $n:\RR\to\RR$ can be decomposed into Fourier modes as
$$
n(\theta) = \sum_{p\geq 0}\widehat n_p\cos(p\theta).
$$
Given $\gamma>0$, define the Sobolev space $H^{\gamma}$ of $1$-periodic even maps $n:\RR\to\RR$ such that the sequence
$$(p^{\gamma}\widehat{n}_p)_p$$ 
is bounded. On this space, we consider the Banach norm
$$
\|n\|_{\gamma} = \sup \left\{\{p^{\gamma}|\widehat n_p|,\,p\geq 1\}\cup\{|\widehat n_0|\}.
\right\}$$

\begin{proposition}[Domains $\mathscr C^7$-close to an ellipse]
\label{proposition:perturbative_regime}
Let $\gamma\in(3,4)$ and $n\in H^{\gamma}$. For any $q\geq 2$, the quantity $a_{\Omega,q}(n)$ can be expressed as
\begin{equation}
\label{equation:perturbative_a_q}
a_{\Omega,q}(n) = a_{\mathscr E,q}(n) + \alpha_0(n) + \frac{\alpha_1(n)}{q^2}+ \frac{\alpha_2(n)}{q^3}+\frac{\|n\|_{\gamma}R_{\Omega,q}(n)}{q^\gamma}
\end{equation}
where $\alpha_0 = \lambda_{\Omega} \widehat n_0$ and $\lambda_{\Omega}=\bigo{\delta}$ depends on $\Omega$. $\alpha_1$ and $\alpha_2$ are first order linear differential operators acting on $n$ and the remainder $R_{\Omega,q}(n)$ is real valued, tending to $0$ uniformly in both $q$ and $n$ as $\Omega$ converges to the ellipse $\mathscr E$ in the $\mathscr C^7$-smooth topology.
\end{proposition}

\begin{remark}
    In Proposition \ref{proposition:perturbative_regime} and throughout the remainder of the paper, we say that a domain $\Omega$ \textit{converges to an ellipse} $\E$ in the $\mathscr C^r$-topology if the affine curvature of $\Omega$ converges in the $\mathscr C^{r-3}$-topology to a constant (see Proposition \ref{Constantcurvature} in Appendix \ref{section:affine_parametrization}).
\end{remark}
The proofs of Propositions \ref{proposition:elliptical_variation} and \ref{proposition:perturbative_regime} follow from estimates given in Appendix \ref{section:proof_estimates}.

\subsection{Isospectral operator and proof of Theorem \ref{theorem:main1_simple}}
In this subsection, we show how the proof of Theorem \ref{theorem:main1_simple} can be deduced from Propositions \ref{proposition:elliptical_variation} and \ref{proposition:perturbative_regime}.
\\
\\
Let $\Omega$ be a strictly convex domain with $\mathscr C^7$-smooth boundary and affine perimeter $1$. Assume that $\partial\Omega$ contains the two points $O$ and $O'$ and that it is symmetric with respect to the line $OO'$. Let us define an operator $T_{\Omega}$ from the space of $1$-periodic even maps $n:\RR\to\RR$ to the space of real valued sequences $(u_q)_{q\geq 0}$. In the notation of Propositions \ref{proposition:elliptical_variation} and \ref{proposition:perturbative_regime}, we write for any integer $q \geq 0$ and $1$-periodic even map $n:\RR\to\RR$ 
\begin{equation}
\label{equation:definition_operator}
T_{\Omega}(n)_q = 
\left\{\begin{array}{ll}
\widehat n_0 & \text{if } q=0;\\
n(0) & \text{if } q=1;\\
n(1/2) & \text{if } q=2;\\
a_{\Omega,q}(n)-(\mu_q+\lambda_{\Omega})\widehat n_0-\frac{\alpha_1(n)}{q^2}-\frac{\alpha_2(n)}{q^3} & \text{if } q\geq 3.
\end{array}\right.
\end{equation}
This defines a sequence $T_{\Omega}(n) = (T_{\Omega}(n)_q)_{q \geq 0}$ of real numbers.
\\
\\
Given $\gamma>1$, we consider the operator $T_{\Omega}$ restricted to the previously defined space of maps $H^{\gamma}$. This space has an analogue for sequences, namely the space $h^{\gamma}$ of sequences $(u_q)_{q\geq0}$ such that $q^{\gamma} u_q$ is bounded. We now show that $T_{\Omega}(H^{\gamma}) \subset h^{\gamma}$ whenever $\gamma\in(3,4)$.

\begin{proposition}
Let $\gamma\in(3,4)$. The map $T_{\Omega}$ defines a bounded operator from $H^{\gamma}$ to $h^{\gamma}$ called \textit{linear isospectral operator}.
\end{proposition}

\begin{proof}
Let $n\in H^{\gamma}$. Assume that $\delta$ is chosen sufficiently small such that if $\partial\Omega$ is $\delta$-close to $\mathscr E$ in the $\mathscr C^7$-smooth topology, then $|R_{\Omega,q}(n)|\leq 1$ for all $q$ and $n$. Hence, by Proposition \ref{proposition:perturbative_regime}, for $q\geq 3$,  $T_{\Omega}(n)_q$ can be expressed as
$$
T_{\Omega}(n)_q = \mu_q[n]_q^{\ast}+\bigo{\frac{\|n\|_{\gamma}}{q^\gamma}},
$$
where $[n]_q^{\ast} = [n]_q-\widehat{n}_0$. As shown in Lemma \ref{lemma:bracket_estimates}, $n\in H^{\gamma}$ implies that $[n]_q^{\ast}\in h^{\gamma}$ together with $\|[n]^{\ast}\|_{\gamma}\leq C\|n\|_{\gamma}$ for some universal constant $C>0$. Hence,
$$
\|T_{\Omega}(n)\|_{\gamma} = \sup_{q>0} q^{\gamma}|T_{\Omega}(n)_q|\leq C\sup_q \mu_q+1<+\infty,
$$
which concludes the proof of the proposition.
\end{proof}

The idea of calling $T_{\Omega}$ a \textit{linear isospectral operator}  goes back to \cite{DKW}, where $T_{\Omega}$ has an analogue for classical (Birkhoff) billiards. The name refers to the following property of $T_{\Omega}$, which is related to length isospectral deformations of $\Omega$:

\begin{proposition}
\label{proposition:kernel_operator}
If $(\Omega_{\tau})_{\tau}$ satisfies the assumptions of Theorem \ref{theorem:main1_simple}, then $T_{\Omega_{\tau}}(n_{\tau}) = 0$ for all $\tau$.
\end{proposition}

\begin{proof}
Fix $\tau$ and write $\Omega=\Omega_{\tau}$ and $n = n_{\tau}$. By Proposition \ref{proposition:vanish_on_n}, for any $q\geq 2$, we can write
\begin{equation}
\label{equation:zero_first_variation}0 = a_{\Omega,q}(n) = a_{\mathscr E,q}(n) + \alpha_0(n) + \frac{\alpha_1(n)}{q^2}+ \frac{\alpha_2(n)}{q^3}+\frac{\|n\|_{\gamma}R_{\Omega,q}(n)}{q^\gamma}.
\end{equation}
Since $\gamma>3$, the terms in $1/q^0$, $1/q^2$ and $1/q^3$ should vanish in Equation \eqref{equation:zero_first_variation}. Using the expressions of $a_{\mathscr E,q}(n)$ given in Proposition \ref{proposition:elliptical_variation} and of $\alpha_0(n)$ given in Proposition \ref{proposition:perturbative_regime}, we deduce that
$$
\widehat n_0 = 0, \qquad \alpha_1(n) = 0,
\qquad\text{and}\qquad
\alpha_2(n) = 0.
$$
This automatically implies that $T_{\Omega}(n)_q = 0$ for $q\geq 3$ and $q=0$. The cases $q=1$ and $q=2$ also hold since each domain $\partial\Omega_{\tau}$ contains the points $O$ and $O'$.
\end{proof}

As stated in Proposition \ref{proposition:kernel_operator}, for each $\tau\in I$, the deformation map $n_{\tau}$ associated to an isospectral deformation lies in the kernel of the operator $T_{\Omega_{\tau}}$. We now prove that in fact, if $\Omega$ is sufficiently close to an ellipse in the $\mathscr C^7$ topology, the operator $T_{\Omega}:H^{\gamma}\to h^{\gamma}$ is invertible and hence, has trivial kernel.

\begin{proposition}
\label{proposition:invertibility_operator}
Let $\gamma\in(3,4)$ and $\mathscr E$ be an ellipse. There exists $\delta=\delta(\mathscr E)>0$ such that if $\Omega$ is $\delta$-close to $\mathscr E$ in the $\mathscr C^r$-topology, then the operator $T_{\Omega}:H^{\gamma}\to h^{\gamma}$ is invertible.
\end{proposition}

\begin{proof}
The proof relies on the two following points: 1) for the ellipse $\mathscr E$, the operator $T_{\mathscr E}:H^{\gamma}\to h^{\gamma}$ is bounded and invertible; 2) the operators $T_{\Omega}$ and $T_{\mathscr E}$ are arbitrarily close to each other in norm if $\delta$ is sufficiently small. The conclusion will follow from 1) and 2).
\\
\\
Invertibility of $T_{\mathscr E}$ is proved in appendix, \textit{cf.} Proposition \ref{proposition:invertibility_ellipse_operator}. To prove 2), let us fix $\varepsilon>0$ and choose $\delta_0>0$ so that if $\Omega$ is $\delta$-close to $\mathscr E$ in the $\mathscr C^7$ topology, then for any $n\in H^{\gamma}$ and $q\geq 3$, we have $|R_{\Omega,q}(n)|\leq \varepsilon$. The formula for $T_{\Omega}(n)_q$ given in Equation \eqref{equation:definition_operator} implies that
$$
q^{\gamma}|T_{\Omega}(n)_q  - T_{\mathscr E}(n)_q| = \bigo{\varepsilon\|n\|_{\gamma}}.
$$
Note that by definition, $T_{\Omega}(n)$ and $T_{\mathscr E}(n)_q$ coincide for $q\in\{0,1,2\}$.
As a consequence, $\|T_{\Omega}  - T_{\mathscr E}\|_{\gamma}=\bigo{\varepsilon}$ in the operator norm $\|\cdot\|_{\gamma}$ associated to the Banach spaces $H^{\gamma}$ and $h^{\gamma}$.
\end{proof}

We can now prove Theorem \ref{theorem:main1_simple}.

\begin{proof}[Proof of Theorem \ref{theorem:main1_simple}]
Let $\delta>0$ be such that for any domain $\Omega$ which is $\delta$-close to $\mathscr E$ in the $\mathscr C^7$-topology, the operator $T_{\Omega}$ is invertible. In particular, if $(\Omega_{\tau})_{\tau}$ is a one-parameter family of domains $\delta$-close to $\mathscr E$, then the operators $T_{\Omega_{\tau}}$ are invertible. Yet for any $\tau\in I$, the deformation map $n_{\tau}$ lies in the kernel of $T_{\Omega_{\tau}}$, as stated in Proposition \ref{proposition:kernel_operator}. Hence, $n_{\tau}$ vanishes identically which, as a consequence of Proposition \ref{proposition:zero_deformation_map}, implies that $\Omega_{\tau} = \Omega_{0}$ for all $\tau$.
\end{proof}

\begin{proof}[Proof of Theorem \ref{theorem:main2_simple}]
For arbitrary $\Omega$, it is not clear whether or not $T_{\Omega}$ is invertible. However, given $q_0>0$, the operator $\tilde T_{\Omega}$ given for any $n$ and any $q\geq q_0$ by 
$$\tilde T_{\Omega}(n)_q = T_{\Omega}(n)_q$$
can be decomposed as 
$$
\tilde T_{\Omega} =  \tilde T_{\Omega}^1+\tilde T_{\Omega}^2,
$$
where $\tilde T_{\Omega}^1$ is defined on maps $n$ whose Fourier coefficients of order $>q_0$ vanish and $\tilde T_{\Omega}^2$ is defined on maps $n$ whose Fourier coefficients of order $\leq q_0$ vanish. This decomposition is obtained by projecting $\tilde{T}_{\Omega}$ onto the $L^2$ orthogonal subspaces spanned by Fourier modes of order greater than and less than or equal to $q_0$.
\\
\\
Previous computations, similar to those in \cite[Theorem 6.1]{DKW}, show that given $\Omega$, there exists $q_0$ such that $\tilde T_{\Omega}^2$ is invertible. By compactness, we can choose $q_0$ so that for any $\tau$, the operator $\tilde T_{\Omega_{\tau}}^2$ is invertible. Isospectral deformations then satisfy
$$
\tilde T_{\Omega_{\tau}}(n_{\tau}) = \tilde T_{\Omega_{\tau}}^1(n_{\tau}^-)+\tilde T_{\Omega_{\tau}}^2(n_{\tau}^+)=0,
$$
where $n_{\tau}^-$ is the projection of $n_{\tau}$ onto the space of maps whose Fourier coefficients of order $>q_0$ vanish and $n_{\tau}^+$ is the projection of $n_{\tau}$ on the space of maps whose Fourier coefficients of order $\leq q_0$ vanish. As a result,
$$
n_{\tau}^+ = F_{\tau}(n_{\tau}^-),
$$
where $F_{\tau} = -{(\tilde T_{\Omega_{\tau}}^2})^{-1} \tilde T_{\Omega_{\tau}}^1$. The theorem then follows from the fact that ${(\tilde T_{\Omega_{\tau}}^2})^{-1}$ is finite rank.
\end{proof}

\appendix
\section{Affine curvature}
\label{section:affine_parametrization}

Consider a strictly convex curve $\Gamma\subset\RR^2$ whose radius of curvature is denoted at each point by $\varrho$. The \textit{affine length} of $\Gamma$ is defined to be
$$
L=\int_{0}^{|\partial\Omega|}\varrho^{-1/3}(s)ds,
$$ 
where $s$ denotes here the regular arclength curvature.
\\
\\
Given two vectors $u,v$ of $\RR^2$, denote by $[u,v]=\omega(u,v)$ the determinant of $u$ and $v$ in a positively oriented orthonormal basis of $\RR^2$; equivalently, $\omega=dx\wedge dy$, where $(x,y)$ are Darboux coordinates corresponding to the canonical basis of $\RR^2$.

\begin{proposition}[See \cite{SapiroTannenbaum}]
$\Gamma$ can be parametrized by a map $\gamma:\RR\to\RR^2$ which is $L$-periodic and satisfies $[\gamma'(t),\gamma''(t)]=1$ for each $t\in\RR$. Such a map is called an \textit{affine arclength parametrization} of $\gamma$.
\end{proposition}

\begin{proof}
Consider an arclength parametrization $\delta(s)$ of $\Gamma$ and define $\gamma(t) = \delta(\varphi(t))$, where $\varphi:\RR\to\RR$ is a diffeormorphism given implicitly by
$$
\varphi^{-1}(s) =\int_0^s\varrho^{-1/3}(\sigma)d\sigma.
$$
\end{proof}

It can be deduced that $[\gamma'(t),\gamma^{(3)}(t)]=0$ for any $t$ and hence, there is a map $k:\RR\to\RR$, called the \textit{affine curvature} of $\Gamma$, which satisfies
$$
\gamma^{(3)}=-k\gamma'.
$$
The following proposition gives formulas for the first three affine arclength derivatives of $\gamma$ which we will use in the paper:

\begin{proposition}
\label{proposition:derivatives_gamma}
Let $u(t)$ be the unit tangent vector to $\partial\Omega$ at $\gamma(t)$ and denote by $N(t)$ be the outward pointing unit normal to $\partial\Omega$ at $\gamma(t)$. Then,
\begin{align*}
    \gamma'(t) &= \varrho^{1/3}(t)u(t),\\
    \gamma''(t) &= \frac{1}{3}\varrho'(t)\varrho(t)^{-2/3}u(t)-\varrho(t)^{-1/3}N(t),\\
    \gamma^{(3)}(t) &= -k(t)\varrho(t)^{1/3}u(t),
\end{align*}
where 
\begin{align*}
  k(t) &= \varrho(t)^{-4/3}-\frac{1}{3}\varrho(t)^{-1/3}(\varrho'\varrho^{-2/3})'(t)\\
  &= \varrho(t)^{-4/3}-\frac{1}{3}\varrho''(t)\varrho^{-1}+\frac{2}{9}\varrho'(t)^2\varrho(t)^{-2}.  
\end{align*}

\end{proposition}

It is known \cite{SapiroTannenbaum} that conic sections $\Gamma$ are characterized uniquely by constant affine curvature $k_\Gamma$. In particular, $k>0$ if $\Gamma$ is an ellipse, $k=0$ if $\Gamma$ is a parabola, and $k<0$ if $\Gamma$ is a hyperbola.

\begin{proposition}\label{Constantcurvature}
Assume that $\partial\Omega$ has affine constant curvature $k>0$. Then it is an ellipse.
\end{proposition}

\begin{proof}
If $\gamma(t)$ is a parametrization of $\partial\Omega$ by affine arclength and $k > 0$ a constant, the following linear differential equation
$$
\gamma^{(3)}(t) = -k\gamma'(t)
$$
can be solved explicitly: there exist vectors $u,v,w\in\RR^2$ such that for all $t\in\RR$
$$
\gamma(t) = \cos(\sqrt{k}t)u+\sin(\sqrt{k}t)v+w.
$$
This implies that $\partial\Omega$ is an ellipse.
\end{proof}
The converse is a simple exercise.

\section{Symplectic billiard in an ellipse}
\label{section:billiard_ellipse}

Let $\mathscr E$ be an ellipse of affine curvature $k>0$ which bounds a domain $\Omega$. Assume that $\mathscr E$ is parametrized with respect to affine arclength by $\gamma(t)$. The following proposition shows that the affine arclength can be considered as an angle coordinate in action angle coordinates for the symplectic billiard map on an ellipse.

\begin{proposition}
\label{proposition:consecutive_points_ellipse}
Given $(t,\varepsilon)\in\RR^2$, the points
$$
\gamma(t-\varepsilon),\gamma(t),\gamma(t+\varepsilon)
$$
 on $\mathscr E$ constitute successive reflection points of the symplectic billiard map on $\Omega$. Namely, the line $\gamma(t-\varepsilon)\gamma(t+\varepsilon)$ is parallel to the tangent line of $\mathscr E$ at $\gamma(t)$.
\end{proposition}

\begin{proof}
By changing coordinates on $\RR^2$, we may assume that $\gamma(t)$ is given by 
$$
\gamma(t) = (a\cos(\sqrt k t),b\sin(\sqrt k t)),
$$
where $a,b>0$ satisfy $ab=k ^{-3/2}$ (a necessary condition for $\gamma$ to be parametrized by affine arclength). We must show that 
$$\det(\gamma(t+\varepsilon)-\gamma(t-\varepsilon),\gamma'(t)) = 0.$$
Indeed, 
\begin{multline}
\label{equation:determinant_ellipse}
\det(\gamma(t+\varepsilon)-\gamma(t-\varepsilon),\gamma'(t))=\\
\left|\begin{matrix}
a\cos(\sqrt k (t+\varepsilon))-a\cos(\sqrt k (t-\varepsilon))&-a\sqrt k\sin(\sqrt k t)\\
b\sin(\sqrt k (t+\varepsilon))-b\sin(\sqrt k (t-\varepsilon))&b\sqrt k\cos(\sqrt k t)\\
\end{matrix}\right|
\end{multline}
To simplify, write ${c}^+$ (resp. $s^+$) instead of $\cos \sqrt k (t+\varepsilon)$ (resp. $\sin \sqrt k (t+\varepsilon)$), ${c}^-$ (resp. $s^-$) in place of $\cos \sqrt k (t-\varepsilon)$ (resp. $\sin \sqrt{k} (t - \varepsilon)$), and $c$ (resp. $s$) for $\cos \sqrt k t$ (resp. $\sin \sqrt{k} t$). Equation \eqref{equation:determinant_ellipse} can then be written as
$$
ab\sqrt k\left(c \cdot(c^+-c^-)
+ s \cdot(s^+-s^-)\right).
$$
Using trigonometric formula 
$$
\cos(x-y)=\cos x\cos y+\sin x\sin y,
$$
we see that the quantity \eqref{equation:determinant_ellipse} equals
$$
ab\sqrt k(\cos(\sqrt k\varepsilon)-\cos(-\sqrt k\varepsilon)) = 0,
$$
from which the proposition follows.
\end{proof}

It follows that periodic symplectic billiard orbits in $\Omega$ equidistribute (with respect to the affine arclength parameter), as described in the following.

\begin{corollary}
\label{proposition:equidistribution_orbits}
Assume that $\Omega$ is bounded by an ellipse which is parametrized by an affine arclength coordinate $t$ on the boundary, vanishing at some point on its major axis. Then for any integer $q>0$, the polygon whose successive vertices are given in affine arclength coordinates by $(t_0,\ldots,t_{q-1})$ with
$$
t_j = \frac{2\pi}{\sqrt{k}}\frac{j}{q}\qquad\qquad j=0,\ldots,q-1
$$
constitutes a $q$-periodic symplectic billiard orbit in $\Omega$.
\end{corollary}

\begin{proof}
For a fixed $q$, consider any integer $j\geq 0$, viewed modulo $q$. Introduce
$$
\varepsilon := \frac{2\pi}{q\sqrt{k}}.
$$ 
Considering also $j+1$ and $j-1$ modulo $q$, we compute 
$$
t_j-\varepsilon = t_{j-1}
\qquad\text{and}\qquad
t_j+\varepsilon = t_{j+1},
$$
so that according to Proposition \ref{proposition:consecutive_points_ellipse}, $\gamma(t_{j-1})$, $\gamma(t_{j})$ and $\gamma(t_{j+1})$ are consecutive points of the symplectic billiard in $\Omega$.
\end{proof}

\section{Estimates of periodic orbits}
\label{section:proof_estimates}

This section is devoted to the proof of Proposition \ref{proposition:perturbative_regime}. Fix a strictly convex bounded axially symmetric domain $\Omega$ with $\class^r$-smooth boundary. The heart of the proof amounts to estimating the quantities $t_k^{(q)}$, which are the affine arclength coordinates of impact points of a $q$-periodic orbit in Equation \eqref{equation:a_q}, defined for all integers $q\geq 2$ and $k \in\{0,\ldots,q-1\}$. In what follows, we assume that the boundary $\partial\Omega$ is parametrized with respect to affine arclength by a curve $\gamma(t)$, $t\in\RR/L\ZZ$, where $L>0$ is the affine perimeter of $\partial\Omega$.

\subsection{Lazutkin coordinates for the symplectic billiard map}

The symplectic billiard map inside $\Omega$ defines a map 
$$
(\overline t_1,\overline t_2)\in\mathcal P\mapsto(\overline t_2,\overline t_3)\in\mathcal P,
$$
where $\overline t_1$, $\overline t_2$ and $\overline t_3$ are the affine arclength coordinates of three consecutive impact points and $\mathcal P$ is the appropriately coordinatized phase space for the symplectic billiard map (see \cite{AlbersTabachnikov} for more details). Following \cite{AlbersTabachnikov}, we consider the pair $(t,\varepsilon)$ defined by
$$t = \overline t_1
\qquad\text{and}\qquad 
\varepsilon=\overline t_2-\overline t_1.$$
It defines new coordinates in which the symplectic billiard map writes as
$$(t,\varepsilon)\mapsto(t_1,\varepsilon_1).$$

\begin{proposition}
\label{proposition:estimates_symplectic_billiard_map}
Assume $\Omega$ is a domain with $\mathscr C^7$-smooth boundary. In $(t,\varepsilon)$-coordinates, the symplectic billiard map admits the following asymptotic expansion
\begin{equation}
\label{equation:billiard_map_asymptotics}
\left\{\begin{array}{rcl}
t_1&=&t+\varepsilon\\
\varepsilon_1&=&\varepsilon+\frac{1}{30}k'(t)\varepsilon^4+R_{\Omega}(\varepsilon)\varepsilon^6\\
\end{array}\right.
\end{equation}
where $k$ is the affine curvature of $\partial\Omega$ and $R_{\Omega}(\varepsilon)$ is a uniformly bounded remainder which converges to $0$ as $\partial\Omega$ converges to an ellipse in the $\mathscr C^7$-topology (i.e., as the affine curvature tends to a positive constant).
\end{proposition}

\begin{proof}
The proof is inspired by \cite{AlbersTabachnikov}. Consider an affine arclength parametrization $\gamma(t)$ of $\partial\Omega$ and assume that the affine perimeter is $L=1$. Let 
$$
t=t_1-\varepsilon,\quad t_1, \text{ and}\quad t_2=t_1+\varepsilon_1
$$
be the parameters of three successive points of reflection on $\partial\Omega$ corresponding to a symplectic billiard orbit. They satisfy
$$
[\gamma(t_1+\varepsilon_1)-\gamma(t_1-\varepsilon),\gamma'(t_1)]=0,
$$
where $[u\,|\,v]$ denotes the determinant of the vectors $u,v \in \RR^2$. 
\\
\\
We first justify regularity properties of the symplectic billiard map in a neighborhood of $\varepsilon=0$ by means of the implicit function theorem applied to a suitably chosen functional. Fix $t_1\in\RR$ and consider the functional $F$ given by 
$$
F_{\gamma}(\varepsilon,\varepsilon_1) = \frac{1}{\varepsilon_1+\varepsilon}[\gamma(t_1+\varepsilon_1)-\gamma(t_1-\varepsilon),\gamma'(t_1)].
$$
Note that if $\gamma$ is $\mathscr C^r$-smooth, then $F_{\gamma}$ is $\mathscr C^{r-1}$-smooth as one can see by writing $\gamma(t_1+\varepsilon_1)-\gamma(t_1-\varepsilon) = (\varepsilon_1+\varepsilon)\int_0^1\gamma'((\varepsilon_1+\varepsilon)\theta+t_1-\varepsilon)d\theta$. Moreover, it satisfies $\partial_{\varepsilon_1}F_{\gamma}(0,0) =-1/2\neq 0$. By the implicit function theorem, there exist a $\mathscr C^{r-1}$-smooth map $\varphi_{\gamma}:I\to\RR$, where $I$ is a small interval containing $0$, such that $F_{\gamma}(\varepsilon,\varepsilon_1) = 0$ if and only if $\varepsilon_1=\varphi_{\gamma}(\varepsilon)$. As a consequence of the implicit function Theorem for Banach spaces, it follows that $\varphi$ is continuous in $(\gamma,\varepsilon)$, where $\gamma$ is considered as an element of the set of $\mathscr C^r$-smooth maps endowed with the corresponding topology. In particular, we have an expansion of the symplectic billiard map $(t,\varepsilon)\mapsto(t_1=t+\varepsilon,\varepsilon_1)$ given by
$$
\varepsilon_1 = a_0+a_1\varepsilon+\ldots+\ldots+a_{r-2}\varepsilon^{r-2}+R_{\Omega}^{r-1}(\varepsilon)\varepsilon^{r-1},
$$
where each $a_k$ is a real number depending on the $k$-jet of $F_{\gamma}$ at $(0,0)$ and $R_{\Omega}^{r-1}(\varepsilon)$ is a real valued remainder depending only on the $(r-1)$-jet of $F_{\gamma}$ restricted to the graph of $\varphi_{\gamma}$. In particular, $R_{\Omega}^{r-1}$ depends on the $r$-jet of $\gamma$ on the graph of $\varphi_{\gamma}$. If $\Omega$ is bounded by an ellipse, then $\varepsilon_1=\varepsilon$. It follows that if $\partial\Omega$ converges to an ellipse in the $\mathscr C^r$-topology, then $R_{\Omega}^{r-1}$ converges uniformly to $0$.
\\
\\
Assume that $r=7$. We now compute $a_0,\ldots,a_5$ by asymptotically expanding $F_{\gamma}$. It is clear that $a_0=0$ by continuity, since the glancing region (with zero angle of reflection) corresponds to fixed points of the symplectic billiard map. To compute $a_1$, let us write
\begin{align*}
\int_0^1\gamma'((\varepsilon_1+\varepsilon)\theta+t_1-\varepsilon)d\theta &=\int_0^1\gamma'(t_1)+((a_1+1)\theta-1)\varepsilon\gamma''(t_1)+\bigo{\varepsilon^2}\\
&=\gamma'(t_1)+\frac{1}{2}(a_1-1)\varepsilon\gamma''(t_1)+\bigo{\varepsilon^2}.
\end{align*}
Since $\det(\gamma'(t_1),\gamma''(t_1))=1$, the equality $F_{\gamma}(\varepsilon,\varepsilon_1)=0$ implies that $a_1=1$. Let us now show that $a_2=a_3=0$ by expanding to fourth order. To simplify, we omit the parameter $t_1$. Note also that $\det(\gamma^{(3)},\gamma')=0$ and hence, it is unnecessary to compute the coefficient multiplying $\gamma^{(3)}$. Using the identity $\int_0^1(2\theta-1)d\theta=\int_0^1(2\theta-1)^3d\theta=0$, we obtain
\begin{multline*}
\int_0^1\gamma'((\varepsilon_1+\varepsilon)\theta+t_1-\varepsilon)d\theta=\\
\int_0^1\gamma'+(2\theta-1)\varepsilon\gamma''
+(\theta a_2\varepsilon^2+\theta a_3\varepsilon^3)\gamma''
+\frac{1}{6}(2\theta-1)^3\varepsilon^3\gamma^{(4)}+b_{\varepsilon}(\theta)\gamma^{(3)}+\bigo{\varepsilon^4}\\
=\gamma'+\frac{1}{2}(a_2\varepsilon^2+a_3\varepsilon^3)\gamma''
+\tilde b_{\varepsilon}\gamma^{(3)}+\bigo{\varepsilon^4},
\end{multline*}
where $b$ and $\tilde b$ are maps depending on $(\varepsilon,\theta)$ and $\varepsilon$ respectively. Considering that $\det(\gamma'',\gamma')=-1$ and $\det(\gamma^{(3)},\gamma')=0$, we now compute
$$
F_{\gamma}(\varepsilon,\varepsilon_1) = -\frac{1}{2}a_2\varepsilon^2-\frac{1}{2}a_3\varepsilon^3+\bigo{\varepsilon^4}.
$$
The condition $F_{\gamma}(\varepsilon,\varepsilon_1)=0$ immediately implies that $a_2=a_3=0$. Let us finally compute $a_4$ and $a_5$ by writing
$$
(\varepsilon_1+\varepsilon)\theta+t_1-\varepsilon = 
t_1+(2\theta-1)\varepsilon+\theta a_4\varepsilon^4+\theta a_5\varepsilon^5
+\bigo{\varepsilon^6}
$$
and then expanding
\begin{multline*}
\int_0^1\gamma'((\varepsilon_1+\varepsilon)\theta+t_1-\varepsilon)d\theta=
\gamma'+\frac{1}{2}(a_4\varepsilon^4+a_5\varepsilon^5)\gamma''+c_{\varepsilon}\gamma^{(3)}\\
+\frac{1}{120}\varepsilon^4\gamma^{(5)}+\bigo{\varepsilon^6},
\end{multline*}
where $c$ is a function depending on $\varepsilon$. Therefore, $F_{\gamma}(\varepsilon,\varepsilon_1)$ is given by
$$
F_{\gamma}(\varepsilon,\varepsilon_1) = -\frac{1}{2}(a_4\varepsilon^4+a_5\varepsilon^5)+\frac{1}{120}\varepsilon^4\det(\gamma^{(5)},\gamma')+\bigo{\varepsilon^6}.
$$
Hence,
$$
a_4=\frac{1}{60}\det(\gamma^{(5)},\gamma')
\qquad\text{and}\qquad
a_5=0.
$$
Differentiating twice the equality $\gamma^{(3)} = -k\gamma$, we obtain $\det(\gamma^{(5)},\gamma') = 2k'$, which implies that
$$
a_4=\frac{1}{30}k',
$$
finishing the proof.
\end{proof}

\subsection{Estimates of $q$-periodic symmetric orbits}

Now consider an axially symmetric domain $\Omega$ whose boundary is parametrized by an affine arclength coordinate $t$ such that $t=0$ corresponds to a point on the axis of symmetry. Given $q\geq 2$, consider the coordinates 
$$(t_0=0,t_1,\ldots,t_{q-1})$$
of consecutive points $p_0,\ldots,p_{q-1}$ on the boundary $\partial\Omega$ such that $p_0\ldots p_{q-1}$ form a polygon of maximal area. We omit the dependence on $q$ of the $t_j$'s and estimate their asymptotic behavior, as well that of each $\varepsilon_j=t_{j+1}-t_j$. These estimates were first computed in \cite[Proposition 4.3]
{BaraccoBernardiNardi}, but we give alternate proof with the additional result that the remainder tends to $0$ for domains which are $\mathscr C^7$-close to ellipses.

\begin{proposition}[See \cite{BaraccoBernardiNardi} Proposition 4.3]
\label{proposition:estimates_glancing_orbit}
Assume that $\partial\Omega$ is $\mathscr C^7$-close to an ellipse $\mathscr E$. Then
$$t_j=L\frac{j}{q}+\frac{1}{q^2}a_0(j/q)+\frac{1}{q^3}a_1(j/q)+\frac{R^0_{\Omega}(j,q)}{q^4},$$
$$\varepsilon_j=\frac{L}{q}+\frac{1}{q^3}b_0(j/q)+\frac{1}{q^4}b_1(j/q)
+\frac{R^1_{\Omega}(j,q)}{q^5},$$
where $L$ is the affine perimeter of $\partial\Omega$, 
$R^0_{\Omega}(j,q)$ and $R^1_{\Omega}(j,q)$ are bounded remainders which converge uniformly to $0$ as $\Omega$ converges to $\mathscr E$ in the $\mathscr C^7$-topology, and
$$
a_0(t) = \frac{L^2}{30}\int_0^{Lt}k(\theta)-(k)_0d\theta,
\qquad
a_1(t) = -\frac{L^3}{30}(k(Lt)-k(0)),
$$
$$b_0(t) = \frac{L^3}{30}(k(Lt)-(k)_0),\qquad b_1(t) = -\frac{L^4}{60}k'(Lt),$$
where $k$ is the affine curvature of $\partial\Omega$ with mean value 
$$(k)_0 = \frac{1}{L}\int_0^Lk(t)dt.$$
\end{proposition}

\begin{remark}
Note that $a_0$ $a_1$, $b_0$ and $b_1$ vanish if $\partial\Omega$ is an ellipse.
\end{remark}


In order to prove Proposition \ref{proposition:estimates_glancing_orbit}, we will need the following lemma on a preliminary bound for $\varepsilon_k$.

\begin{lemma}
\label{lemma:rough_estimates_periodic}
There exist $C=C(\Omega)>0$ such that for any $k\in\{0,\ldots,q-1\}$, the quantity $\varepsilon_k$ is bounded by
$$\varepsilon_k \leq \frac{C}{q}.$$
\end{lemma}

\begin{proof}
This proof is based on an analogue proof for classical billiards found in \cite{AvilaKaloshinSimoi}. Note however that some arguments there are missing. We will provide them here. As a consequence of Equation \ref{equation:billiard_map_asymptotics}, one obtains
$$
1=\varepsilon_0+\ldots+\varepsilon_{q-1}
$$
and hence, there is $k_{\ast}\in\{0,\ldots,q-1\}$ such that 
$$\varepsilon_{k_{\ast}}\leq \frac{1}{q}.$$
Since the orbit is cyclic, without loss of generality, we can assume that $k_{\ast}=0$. Again using Equation \ref{equation:billiard_map_asymptotics}, we deduce the existence of an $M=M(\Omega)>1$ such that for any $k\in\{0,\ldots,q-1\}$,
$$
|\varepsilon_{k+1}-\varepsilon_k|\leq M\varepsilon_{k}^4.
$$
Let us choose $q>0$ sufficiently large; in fact $q\geq q_0:=e^{3/2}M^2$ is enough. By induction on $k$, we see that for any $k\in\{0,\ldots,q-1\}$,
$$
\varepsilon_k\leq \frac{M}{q}\left(1+\frac{1}{q}\right)^k.
$$
It is obviously true for $k=0$ and if it is true for $k$, we estimate
$$
\varepsilon_{k+1}\leq\varepsilon_k+|\varepsilon_{k+1}-\varepsilon_k| \leq \varepsilon_{k}+M\varepsilon_{k}^4.
$$
Hence,
$$
\varepsilon_{k+1}\leq \frac{M}{q}\left(1+\frac{1}{q}\right)^k\left(1+\frac{M^4}{q^3}\left(1+\frac{1}{q}\right)^{3k}\right),
$$
where
$$
\frac{M^4}{q^3}\left(1+\frac{1}{q}\right)^{3k}\leq \frac{M^4e^3}{q^3}\leq \frac{1}{q}
$$
by our assumptions on $q$. This concludes the inductive step. In particular, the statement of Lemma \ref{lemma:rough_estimates_periodic} is true with $C=eM$ for $q\geq q_0$. To make the proof work for all $q$, we can take $C$ larger so that $\varepsilon_k\leq C/q$ is satisfied even if $q<q_0$; the latter only requires a finite set of inequalities to be satisfied.
\end{proof}

\begin{proof}[Proof of Proposition \ref{proposition:estimates_glancing_orbit}]
To prove the result, we will assume that the affine perimeter of $\Omega$ is $L=1$. The general case can be reduced this by homotethy of the domain. Substituting the inequality given in Lemma \ref{lemma:rough_estimates_periodic} into the expansion \eqref{equation:billiard_map_asymptotics}, we see that 
$$
\varepsilon_j=\varepsilon_0+\bigo{\frac{1}{q^3}}
$$ 
uniformly in $j$. Moreover, from the equality 
\begin{equation}
\label{equation:sum_epsilon}
\varepsilon_0+\ldots+\varepsilon_{q-1}=1,
\end{equation}
we deduce that
$$\varepsilon_j=\frac{1}{q}+\bigo{\frac{1}{q^3}}
\qquad\text{and}\qquad t_j = \frac{j}{q}+\bigo{\frac{1}{q^2}}$$
uniformly in $j$. Substituting again the two previous equations into \eqref{equation:billiard_map_asymptotics}, we obtain the existence of
$1$-periodic maps $a_0,a_1,b_0$ and $b_1$, with $a_0(0)=a_1(0)=0$ and bounded remainders $R^0_{\Omega}(j,q)$, $R^1_{\Omega}(j,q)$ which converge to $0$ uniformly as $\Omega$ converges to $\mathscr E$ in the $\mathscr C^7$-topology such that
$$
t_j=\frac{j}{q}+\frac{1}{q^2}a_0(j/q)+\frac{1}{q^3}a_1(j/q)+\frac{R^0_{\Omega}(j,q)}{q^4},
$$
$$
\varepsilon_j=\frac{1}{q}+\frac{1}{q^3}b_0(j/q)+\frac{1}{q^4}b_1(j/q)+\frac{R^1_{\Omega}(j,q)}{q^5}.
$$
By definition, the points of reflection satisfy
$$
\varepsilon_j = t_{j+1}-t_j,
$$
which implies that 
\begin{equation}
\label{equation:on_a}
a_0' = b_0
\qquad\text{and}\qquad
a_1'=b_1-\frac{1}{2}b_0'.
\end{equation}
We now compute $\varepsilon_{j+1}-\varepsilon_j$ in two different ways, first as
$$\varepsilon_{j+1}-\varepsilon_j = \frac{b_0'(j/q)}{q^4}+\frac{1}{q^5}\left(\frac{b_0''(j/q)}{2}+b_1'(j/q)\right)+\bigo{\frac{1}{q^6}}$$
and second, using the billiard map expansion given by Proposition \ref{proposition:estimates_symplectic_billiard_map}, as
$$\varepsilon_{j+1}-\varepsilon_j = \alpha(t_j)\varepsilon_j^4+\bigo{\varepsilon_j^6}=\frac{\alpha(j/q)}{q^4}+\bigo{\frac{1}{q^6}},$$
where $\alpha(t) = \frac{1}{30}k'(t)$. This leads to the equations 
\begin{equation}
\label{equation:on_b}
b_0'=\alpha
\qquad\text{and}\qquad
b_1'=-\frac{1}{2}b_0'' = -\frac{1}{2}\alpha'.
\end{equation}
The lemma follows by integrating Equations \eqref{equation:on_a} and \eqref{equation:on_b} and taking into account periodicity and initial conditions on the maps.
\end{proof}

We deduce the following asymptotic estimates.

\begin{lemma}
Assume that $\partial\Omega$ is $\mathscr C^7$-close to an ellipse $\mathscr E$ of affine perimeter $L_{\mathscr E}$ and affine curvature $k_{\mathscr E}$. We then have the following estimates for all $q>1$ and $j\in\{0,\ldots,q-1\}$:
\begin{equation}
\label{equation:estimates_rho_L}
\varrho\left(t_j\right)^{-1/3}L_j =
2k_{\mathscr E}^{-1/2}\sin\left(\frac{2\pi}{q}\right)+\frac{c_0}{q}+\frac{c_1(Lj/q)}{q^3}+\frac{c_2(Lj/q)}{q^4}+\frac{R_{\Omega}(j,q)}{q^5}
\end{equation}
where $R_{\Omega}(j,q)$ is a uniformly bounded remainder which converges to $0$ as $\Omega$ converges to $\mathscr E$ in the $\mathscr C^7$-topology,
\begin{align*}
    c_0 &= 2(L-L_{\mathscr E}),\\
    c_1(\theta) &= \frac{L^3}{3}(k_{\mathscr E}L_{\mathscr E}^3-k(\theta))+\frac{L^3}{15}L(k(\theta)-(k)_0)),\\
    c_2(\theta) &= -\frac{L^4}{15}k'(\theta),
\end{align*}
$L$ is the affine perimeter of $\partial\Omega$ and $k$ is its affine curvature.
\end{lemma}

\begin{proof}
In the following, the exponent or index $\E$ will indicate that the quantity is related to the dynamics in the ellipse. We assume that the ellipse is parametrized by
$$
\gamma_\E(t) = (a\cos(\sqrt{k_{\mathscr E}}t),b\sin(\sqrt{k_{\mathscr E}}t)),
$$
where $a,b>0$ and $ab=k_{\mathscr E}^{-3/2}$. We compute its radius of curvature
$$
\varrho_\E(t) = k_{\mathscr E}^{3/2}(a^2\sin^2(\sqrt{k_{\mathscr E}}t)+b^2\cos^2(\sqrt{k_{\mathscr E}}t))^{3/2}
$$
and the following quantity:
\begin{multline}
\label{equation:factor_ellipse}
\varrho_\E(t_j^\E)^{-1/3}L_j^\E = \varrho_\E(t_j^\E)^{-1/3}\|\gamma_\E(t_{j+1}^\E)-\gamma_\E(t_{j-1}^\E)\| \\
= 2k_{\mathscr E}^{-1/2}\sin\left(\frac{2\pi}{q}\right)
= \frac{2}{q}L_{\mathscr E}-\frac{1}{3q^3}k_{\mathscr E}L_{\mathscr E}^3+\bigo{\frac{1}{q^5}}.
\end{multline}
Now consider a parametrization $\gamma$ of $\partial\Omega$ which is $\delta$-$\mathscr C^5$-close to $\gamma_0$. For simplicity, given an $L$-periodic map $f:\RR\to\RR$, we write $f=f(Lj/q)$, $f^+=f(L(j+1)/q)$ and $f^-=f(L(j-1)/q)$. Then 
$$
\gamma(t_{j+1})-\gamma(t_{j-1}) = \gamma^+-\gamma^-+\frac{2(a_0'\gamma'+La_0\gamma'')}{q^3}+\frac{2(a_1'\gamma'+La_1\gamma'')}{q^4}+\frac{R_q}{q^{5}},
$$
where $R_q$ is a functional depending on $\gamma^{(4)}$ and $a_0,a_1$ are given in Proposition \ref{proposition:estimates_glancing_orbit}. We first estimate
$$
\gamma^+-\gamma^- = \left(\frac{2L}{q}-\frac{kL^3}{3q^3}\right)\gamma'+\frac{R_q'}{q^{5}},
$$
where $R_q'$ is a functional depending on $\gamma^{(5)}$. Moreover, if $a$ is an $L$-periodic map,
$$
a'\gamma'+La\gamma'' = (a'+\frac{L}{3}\varrho'\varrho^{-1}a)\gamma'-La\varrho^{-1/3}N,
$$
where $N$ is the unit outward normal vector to $\partial\Omega$ at the corresponding point on $\d \Omega$. The latter equation follows from Proposition \ref{proposition:derivatives_gamma}. From the preceding equations, we deduce that
\begin{multline}
\gamma(t_{j+1})-\gamma(t_{j-1}) = \left(\frac{2L}{q}-\frac{kL^3}{3q^3}+\frac{1}{q^3}(a_0'+\frac{L}{3}\varrho'\varrho^{-1}a_0)+\frac{1}{q^4}(a_1'+\frac{L}{3}\varrho'\varrho^{-1}a_1)\right)\gamma'\\
+\bigo{\frac{1}{q^3}}N.
\end{multline}
Hence, $\|\gamma'\|=\varrho^{1/3}$ (cf. Proposition \ref{proposition:derivatives_gamma}) and we obtain
$$
L_j = \frac{2L}{q}\varrho^{1/3}+\frac{\varrho^{1/3}}{q^3}\left(2\left(a_0'+\frac{L}{3}\varrho'\varrho^{-1}a_0\right)-\frac{kL^3}{3}\right)+\frac{\varrho^{1/3}}{q^4}\left(a_1'+\frac{L}{3}\varrho'\varrho^{-1}a_1\right)+\frac{R_q''}{q^{5}},
$$
where $R_q''$ is a functional depending on $\gamma^{(5)}$. By Proposition \ref{proposition:estimates_glancing_orbit}, the following expansion holds:
$$\varrho(t_j)^{-1/3} = \varrho^{-1/3}-\frac{1}{3}\varrho'\varrho^{-4/3}\frac{a_0}{q^2}-\frac{1}{3}\varrho'\varrho^{-4/3}\frac{a_1}{q^3}+\bigo{\frac{1}{q^4}}.$$
Therefore,
$$\varrho(t_j)^{-1/3}L_j = 
\frac{2L}{q}+\frac{1}{3q^3}(6a_0'-kL^3)+\frac{2a_1'}{q^3}+\frac{r_q}{q^{5}},$$ 
where $r_q$ is a functional depending on $\gamma^{(5)}$. The expressions for $c_j$ follow from the previous expansion and estimates on each asymptotic term when $\partial\Omega$ is $\mathscr C^7$-close to an ellipse, using Equation \eqref{equation:factor_ellipse} and the expressions for $a_0$ and $a_1$ given in Proposition \ref{proposition:estimates_glancing_orbit}.
\end{proof}

We can now present the

\begin{proof}[Proof of Proposition \ref{proposition:perturbative_regime}]
Let $\Omega$ be a domain which is close to an ellipse $\mathscr E$ in the $\mathscr C^7$-smooth topology. For any integer $q\geq 3$ and $j\in\{0,\ldots,q-1\}$, the estimates on the $j$-th impact point $t_j$ of a nearly glancing $q$-periodic orbit given by Proposition \ref{proposition:estimates_glancing_orbit} imply that
\begin{equation}
\label{equation:expansion_n}
n(t_j) = n(Lj/q) + n'(Lj/q)\frac{a_0(j/q)}{q^2} + n'(Lj/q)\frac{a_1(j/q)}{q^3} + \frac{\|n\|_{\mathscr C^2}R_{\Omega}(j,q)}{q^4},
\end{equation}
where $R_{\Omega}(j,q)$ is a uniformly bounded remainder which converges to $0$ as $\Omega$ converges to $\mathscr E$ in the $\mathscr C^7$ topology. Multiplying Equation \eqref{equation:estimates_rho_L} with Equation \eqref{equation:expansion_n} gives
\begin{multline}
\label{equation:summand_a_q}
n(t_j)\varrho\left(t_j\right)^{-1/3}L_j =
2k_{\mathscr E}^{-1/2}\sin\left(\frac{2\pi}{q}\right)n(Lj/q)+\frac{c_0}{q}n(Lj/q)\\
+\frac{c_1(Lj/q)}{q^3}n(Lj/q)
+\frac{(4\pi k_{\mathscr E}^{-1/2}+c_0)a(Lj/q)}{q^3}n'(Lj/q)\\
+\frac{c_2(Lj/q)}{q^4}n(Lj/q)
+\frac{(4\pi k_{\mathscr E}^{-1/2}+c_0)a_1(Lj/q)}{q^4}n'(Lj/q)
\\+\frac{\|n\|_{\mathscr C^2}R^{\ast}_{\Omega}(j,q)}{q^5},
\end{multline}
where $R^{\ast}_{\Omega}(j,q)$ is a uniformly bounded remainder which converges uniformly to $0$ as $\Omega$ converges to $\mathscr E$ in the $\mathscr C^7$ topology. 
\\
\\
We now sum over $j$ in the equation \eqref{equation:summand_a_q}. First, note that
Lemmas \ref{lemma:bracket_estimates} and \ref{lemma:norm_product} imply that 
$$\frac{1}{q}\sum_{j=0}^{q-1}c(Lj/q)n(Lj/q) = (cn)_0 + \bigo{\frac{\|n\|_{\gamma}}{q^{\gamma}}}$$
and
$$\frac{1}{q}\sum_{j=0}^{q-1}c(Lj/q)n'(Lj/q) = (cn')_0 + \bigo{\frac{\|n\|_{\gamma}}{q^{\gamma-1}}}$$
for any $L$-periodic map $c$ in the space $H^{\gamma}$, where $(\cdot)_0$ stands for the average of the considered map. Following this remark, we conclude the proof by setting
$$\alpha_0(n) = (c_0n)_0=c_0\widehat n_0,$$
$$\alpha_1(n) = (c_1n+(4\pi k_{\mathscr E}^{-1/2}+c_0)a_0n')_0,$$
and
$$
\alpha_2(n) = (c_2n+(4\pi k_{\mathscr E}^{-1/2}+c_0)a_1n')_0.
$$
\end{proof}

%

\section{Operators acting on Sobolev spaces}
\label{section:operators_appendix}

For any $\gamma>0$, denote by $h^{\gamma}$ the space of sequences $(u_q)_{q\geq 0}$ such that 
$$
q^{\gamma}u_q=\bigo{1}
$$
and define the norm $\|u\|_{\gamma} = \sup\{q^{\gamma}|u_q|,\,\, q\in\NN\}\cup\{|u_0|\}$. Given a continuous $1$-periodic even map $n:\RR\to\RR$, consider its Fourier decomposition 
$$
n=\sum_{j\geq 0}\widehat n_j e_j,
$$
where $e_j$ denotes the $j$th Fourier mode given by $e_j(\theta) = \cos(2\pi j\theta)$ for any $\theta\in\RR$. The space of even $1$-periodic maps $n:\RR\to\RR$ such that 
$$q^{\gamma}\widehat n_q=\bigo{1}$$
will be denoted by $H^{\gamma}$ with the norm $\|\cdot\|_{\gamma}$ defined by
$$\|n\|_{\gamma} = \sup\left\{\{q^{\gamma}|\widehat n_q|,\,\, q\in\NN\}\cup\{|\widehat n_0|\}\right\}.$$ 
Note that  we can extend $\|n\|_{\gamma}$ to $\RR\cup\{+\infty\}$ when the $\sup$ involved in the definition is not finite. In this case, $n\in H^{\gamma}$ if and only if $\|n\|_{\gamma}<+\infty$.
\\
\\
We can define analogous notions for odd maps, by considering their decomposition in the $\sin$ basis. We will sometimes identify $h^{\gamma}$ with $H^{\gamma}$ via the isometry
$$n=\sum_{j\geq 0}\widehat n_je_j\in H^{\gamma}\leftrightarrow (\widehat n_j)_j\in h^{\gamma}.$$
Given a $1$-periodic even map $F:\RR\to\RR$, define for any $q\geq 1$ the quantities
$$[F]_q = \frac{1}{q}\sum_{k=0}^{q-1}F(k/q)
\qquad\text{and}\qquad
[F]_q^{\ast} = [F]_q-\widehat F_0.$$

\begin{remark}
\label{remark:bracket_and_Fourier}
$[F]_q$ and $[F]_q^{\ast}$ are related to the Fourier coefficients of $F$ whose order is divisible by $q$, namely
$$[F]_q = \sum_{m\geq 0} \widehat F_{mq}
\qquad\text{and}\qquad [F]_q^{\ast} = \sum_{m>0} \widehat F_{mq}.$$
\end{remark}

\begin{lemma}
\label{lemma:bracket_estimates}
Assume that $F\in H^\gamma$ with $\gamma>1$. Then we have the following bound on $[F]_q^{\ast}$:
$$
|[F]_q^{\ast}|\leq \frac{\zeta(\gamma)}{q^{\gamma}}\|F\|_{\gamma},
$$
where $\zeta(\gamma)$ is the Riemann $\zeta$ function evaluated at $\gamma$.
\end{lemma}

\begin{proof}
The Fourier decomposition of $F$ induces the following expansion
$$[F]_q = \sum_{m\geq 0} \widehat F_{mq} = \widehat F_0+\sum_{m> 0} \widehat F_{mq}.$$
By definition, we have the bound
$$
|\widehat F_{mq}|\leq \frac{1}{(mq)^{\gamma}}\|F\|_{\gamma}.
$$
The estimates follow by summing over $m$ and recalling that $\zeta(\gamma) = \sum_{m \geq 0} m^{-\gamma}.$
\end{proof}

\begin{lemma}
\label{lemma:norm_product}
Let $\gamma>1$ and $\alpha,\beta$ be $1$-periodic maps of the same parity. Then
$$\|\alpha\beta\|_{\gamma}\leq C_{\gamma}\|\alpha\|_{\gamma}\|\beta\|_{\gamma},$$
where $C_{\gamma}=2(\zeta(\gamma)+1)>0$.
\end{lemma}

\begin{proof}
Without loss of generality, assume that $\alpha$ and $\beta$ are even. Computing the $j$-th Fourier coefficient of $\alpha\beta$, we find
$$
\widehat{(\alpha\beta)}_j = \frac{1}{2}\sum_{k=0}^j\widehat \alpha_k\widehat \beta_{j-k}+\frac{1}{2}\sum_{|k-\ell|=j}\widehat \alpha_k\widehat \beta_{\ell}.
$$
To conclude, we use the bound
$$
\sum_{k=1}^{j-1}\frac{1}{k^{\gamma}(j-k)^{\gamma}}\leq \frac{2\zeta(\gamma)}{j^{\gamma}},
$$
which comes from the inequality
$$
\frac{1}{x^{\gamma}(1-x)^{\gamma}}\leq \frac{1}{x^{\gamma}}+\frac{1}{(1-x)^{\gamma}}
$$
for any $x\in(0,1)$.
\end{proof}

\begin{lemma}
\label{lemma:link_derivative_sobolev}
Let $r>1$ and $\gamma>r+1$. Assume that $n:\RR\to\RR$ is a $1$-periodic even map such that $n\in H^{\gamma}$. Then $n$ is $\mathscr C^r$-smooth and
$$\|n\|_{\mathscr C^r}\leq (2\pi)^r\zeta(\gamma-r)\|n\|_{\gamma}.$$
\end{lemma}

\begin{proof}
Expanding $n$ in Fourier modes, we see that for any integer $0\leq k\leq r$,
$$\|n^{(k)}\|_{\mathscr C^0}\leq\sum_{j\geq0}(2\pi j)^k|\widehat n_j|\leq (2\pi)^k \sum_{j\geq 0}j^{k-\gamma}\|n\|_{\gamma}=(2\pi)^k \zeta(\gamma-k)\|n\|_{\gamma}.$$
The result follows by taking the supremum over all possible $k$.
\end{proof}

\begin{proposition}
\label{proposition:invertibility_ellipse_operator}
Let $\gamma>0$ and consider the following operator. For $n\in H^{\gamma}$, let $T(n) = (u_q)_{q\geq 0}$ be the sequence defined by 
\begin{equation}
\label{equation:definition_operator_ellipse}
u_q = 
\left\{\begin{array}{ll}
\widehat n_0 & \text{if } q=0;\\
n(0) & \text{if } q=1;\\
n(1/2) & \text{if } q=2;\\
\mu_q[n]_q^{\ast} & \text{if } q>2,
\end{array}\right.
\end{equation}
where $\mu_q\in\RR^{\ast}$ has a nonzero limit as $q \to \infty$. Then $T$ defines a bounded invertible operator between $H^{\gamma}$ and $h^{\gamma}$.
\end{proposition}

\begin{proof}
We first prove that $T$ is bounded. This comes from estimates on the different expressions of $u_q$: if $q=0$, then $|u_0|\leq \|n\|_{\gamma}$ by definition of the $H^\gamma$ norm. Moreover, by expanding $n$ in Fourier modes, one sees that $|1^{\gamma}u_1|=|n(0)|\leq\zeta(\gamma)\|n\|_{\gamma}$. Finally, from Lemma \ref{lemma:bracket_estimates}, we conclude that $q^{\gamma}|u_q|\leq \zeta(\gamma)\|n\|_{\gamma}$ for any $q\geq 2$.
\\
\\
To prove that $T$ is invertible, we first note that $\mu_q=1$ for any $q\geq 3$, which follows from rescaling the operator. Let us now invert the operator. Fix $u = (u_q)_q \in h^{\gamma}$ and consider the so-called \textit{M\"obius function} $\mu:\ZZ^{>0}\to\{-1,0,1\}$, which is defined as follows. Given an integer $k\geq 0$, consider its decomposition into $s$ distinct primes and set
$$
\mu(k) = \left\{\begin{array}{cl}
0 & \text{if }k\text{ has a square in its decomposition}\\
1 & \text{if }k\text{ has a no squares in its decomposition and }s\text{ is even}\\
-1 & \text{if }k\text{ has a no squares in its decomposition and }s\text{ is odd.}
\end{array}\right.$$
It is known to satisfy the following formula
\begin{equation}
\label{equation:mobius}
\sum_{\substack{d|k}}\mu(d) = \delta_{k,1},
\end{equation}
where the sum ranges over all divisors $d>0$ of $k$ and $\delta_{k,1}$ is the Kronecker delta. For $j>2$, define
$$
\widehat n_j = \sum_{\substack{q>0\\j|q}}\mu\left(\frac{q}{j}\right)u_q.
$$
Set $\widehat n_0=u_0$ and $\widehat n_1,\widehat n_2\in\RR$ so that they satisfy
$$
\sum_{j\geq 0}\widehat n_j=u_1
\qquad\text{and}\qquad
\sum_{j\geq 0}(-1)^j\widehat n_j=u_2.
$$
Defining $n$ to be the even map whose Fourier coefficients are given by the sequence of $\widehat n_j$'s, one can easily show that $n$ lies in $H^{\gamma}$, with $\|n\|_{\gamma}\leq \zeta(\gamma)\|u\|_{\gamma}$. Furthermore, Equation \eqref{equation:mobius} together with the choices of $\widehat n_0$, $\widehat n_1$ and $\widehat n_2$ implies that  $T(n)=u$, which concludes the proof.
\end{proof}

\end{document}